\documentclass[11pt]{article}
\usepackage{amsthm}
\usepackage{amsmath}
\usepackage{amssymb}
\usepackage{latexsym}
\usepackage{amsfonts}
\usepackage{mathrsfs}
\usepackage{listings}
\newtheorem{theorem}{\bf Theorem}[section]

\newtheorem{lemma}[theorem]{\bf Lemma}
\begin{document}
\title{On Diophantine equations involving intersection of Thabit and Williams numbers base $b$ and some ternary recurrent sequences}
\author{Bibhu Prasad Tripathy, Asutosh Satapathy, Utkal Keshari Dutta, Bijan Kumar Patel}
\date{}
\maketitle
\begin{abstract} \noindent
 Let $\mathcal{P}_{n}$ be the $n$-th Padovan number, $E_{n}$ be the $n$-th Perrin number and $N_{n}$ be the $n$-th Narayana's cows number. Let $b$ be a positive integer such that $b \geq 2$. In this paper, we study the Diophantine equations 
 \[
 \mathcal{P}_{n} = (b \pm 1)\cdot b^{l} \pm 1, 
 \]
 \[
 E_{n} = (b \pm 1)\cdot b^{l} \pm 1, 
 \]
and 
\[
 N_{n} = (b \pm 1)\cdot b^{l} \pm 1, 
 \] 
in non-negative integers $n, b$ and positive integer $l$. As a result, we determine the Padovan, Perrin and Narayana's cows numbers that are Thabit and Williams numbers base $b$. Moreover, we determine all solutions of the above equations within the range $2 \leq b \leq 10$. 
\end{abstract}
	
\section{Introduction}
The Padovan sequence $\{P_n\}_{n\geq 0}$ and the Perrin sequence $\{E_n\}_{n\geq 0}$ are recursively defined by the following ternary recurrences
\[
P_{0} = P_{1} = P_{2}= 1, \quad P_{n+3} = P_{n+1} + P_{n}, \quad \text{for all} \quad n \geq 0,
\] 
and
\[
E_{0} = 3, E_{1}=0, E_{2} = 2, \quad E_{n+3} = E_{n + 1} + E_{n}, \quad \text{for all} \quad n \geq 0
\]
respectively. The Padovan and Perrin sequences 
are the sequences  A000931  and A001608 respectively,
in the online encyclopedia of integer sequences (OEIS). The first few terms of the Padovan sequence are
\[
1, 1, 1, 2, 2, 3, 4, 5, 7, 9, 12, 16, 21, 28, 37, 49, 65, 86, 114, 151, 200, \cdots,
\]
while the Perrin sequence begins with
\[
3, 0, 2, 3, 2, 5, 5, 7, 10, 12, 17, 22, 29, 39, 51, 68, 90, 119, 158, 209, 277,  \cdots
\]
Another well-known ternary sequence is the Narayana's cows sequence $\{N_{m}\}_{m \geq 0}$, defined by 
\[
	N_{m+3} = N_{m+2} + N_{m}~~ {\rm for} ~m \geq 0,
\]
with initials $N_{0} = N_{1} = N_{2} = 1$. This sequence, listed as A000930 in the OEIS, begins with
	\[
        1, 1, 1, 2, 3, 4, 6, 9, 13, 19, 28, 41, 60, 88, 129, 189, 277, 406, 595, 872, \dots
	\]
Together with these recurrence sequences, a classical family of numbers introduced by Th${\rm \hat{a}}$bit ibn Qurra has attracted significant attention.

A Thabit number (also called a \emph{321 number}) is defined by $3\cdot2^{n} - 1$ for $n \geq 0$ and is of $n+2$ digits long. This sequence corresponds to A055010 in the OEIS, with the first few terms are
    
\[
    2,5,11,23,47,95,191,383,767,1535,3071,6143, \ldots
\]
Motivated by their structure, generalized forms of Thabit numbers have been studied extensively. For an integer $b \geq 2$, a Thabit number of the first kind base $b$ is a number of the form 
\[
(b+1)b^{l}-1~~{\rm for}~~l \geq 1,
\]
while a Thabit number of the second kind base $b$ is a number of the form 
\[
(b+1)b^{l} + 1~~{\rm for}~~l \geq 1.
\] 
A closely related family is that of Williams numbers, which form another generalization defined for $b \geq 2$ by
\[
(b - 1)b^{l} \pm 1~~{\rm for}~~l \geq 1.
\]
For example, Williams numbers with base $b=2$ reduce precisely to Mersenne numbers.

In recent years, several papers have addressed Diophantine equations involving Thabit and Williams numbers in base $b$. For instance, Adédji et al. \cite{Adedji} studied equations where such numbers appear as sums or differences of two Lucas sequence terms. In a subsequent work \cite{Adedji1}, they investigated whether Thabit and Williams numbers in base $b$ can be expressed as sums or differences of Fibonacci and Multau numbers, and vice versa. More recently, Adédji et al. \cite{Adedji2} examined representations of Thabit and Williams numbers in base $b$ as sums or differences of two $g$-repdigits. Motivated by these studies, we focus on the interaction between Thabit and Williams numbers in base $b$ and the Padovan, Perrin, and Narayana’s cows sequences. Specifically, we establish the following results.
\begin{theorem}\label{thm1}
Let $b \geq 2$ be an integer. Then, the Diophantine equation
     \begin{equation}\label{eq 1.1}
        \mathcal{P}_{n} = (b \pm 1)\cdot b^{l} \pm 1,     
     \end{equation}
has only finitely many solutions in integers $(n, b, l)$. In particular, we have
\[
     n <  1.96\times 10^{13}\log b \log (184 b^3) (29.91+\log(\log b)+\log (\log (184b^3))).
\]
\end{theorem}

\begin{theorem}\label{thm2}
Let $b \geq 2$ be an integer. Then, the Diophantine equation
       \begin{equation}\label{eq 1.2}
        E_{n} = (b \pm 1)\cdot b^{l} \pm 1,     
        \end{equation}
        has only finitely many
solutions in integers $(n, b, l)$. In particular, we have
\[
       n < 1.96 \times 10^{13}\log b (\log (8b^3))(29.91+\log(\log b)+\log (\log (8b^3))).
        \]
	\end{theorem}
     \begin{theorem}\label{thm3}
		Let $b \geq 2$ be an integer. Then, the Diophantine equation
       \begin{equation}\label{eq 1.3}
        N_{n} = (b \pm 1)\cdot b^{l} \pm 1,     
        \end{equation}
        has only finitely many
solutions in integers $(n, b, l)$. In particular, we have
        \[
       n < 1.95 \times 10^{13} \log b \log (248 b^3) (29.91 + \log(\log b) + \log (\log 248b^3)).
        \]
	\end{theorem}

The proofs of our theorems rely primarily on linear forms in logarithms of algebraic numbers, as developed by Matveev \cite{Matveev}, together with the reduction algorithm of Dujella and Pethő \cite{Dujella}. We begin by presenting several preliminary results, which are discussed in detail in the following section.
 
\section{Preliminary results}
This section is devoted to gathering several definitions, notations, properties, and results that will be used in the rest of this study.
\subsection{Linear forms in logarithms}
Let $\gamma$ be an algebraic number of degree $d$ with a minimal primitive polynomial 
\[
f(Y):= b_0 Y^d+b_1 Y^{d-1}+ \cdots +b_d = b_0 \prod_{j=1}^{d}(Y- \gamma^{(j)}) \in \mathbb{Z}[Y],
\]
where the $b_j$'s are relatively prime integers, $b_0 >0$, and the $\gamma^{(j)}$'s are conjugates of $\gamma$. Then the \emph{logarithmic height} of $\gamma$ is given by
\[
	h(\gamma)=\frac{1}{d}\left(\log b_0+\sum_{j=1}^{d}\log\left(\max\{|\gamma^{(j)}|,1\}\right)\right).
\]
With the above notation, Matveev (see  \cite{Matveev} or  \cite[Theorem~9.4]{Bugeaud}) proved the following result.
	
\begin{theorem}\label{thm4}
Let $\eta_1, \ldots, \eta_s$ be positive real algebraic numbers in a real algebraic number field $\mathbb{L}$ of degree $d_{\mathbb{L}}$. Let $a_1, \ldots, a_s$ be non-zero  integers such that
\[
	\Lambda :=\eta_1^{a_1}\cdots\eta_s^{a_s}-1 \neq 0.
\]
Then
\[
	- \log  |\Lambda| \leq 1.4\cdot 30^{s+3}\cdot s^{4.5}\cdot d_{\mathbb{L}}^2(1+\log d_{\mathbb{L}})(1+\log D)\cdot B_1 \cdots B_s,
\]
where
\[
		D\geq \max\{|a_1|,\ldots,|a_s|\},
\]
and
\[
	B_j\geq \max\{d_{\mathbb{L}}h(\eta_j),|\log \eta_j|, 0.16\}, ~ \text{for all} ~ j=1,\ldots,s.
\]
\end{theorem}
	
\subsection{The reduction method}
	Our next tool is a version of the reduction method of Baker and Davenport (see \cite{Baker}). Here, we use a slight variant of the version given by Dujella and Peth\"{o} (see \cite{Dujella}). For a real number $x$, we write $||x||$ for the distance from  $x$ to the nearest integer.
	
	\begin{lemma}\label{lem 2.2}
		Let $M$ be a positive integer, $p/q$ be a convergent of the continued fraction of the irrational $\tau$ such that $q > 6M$, and $A, B, \mu $ be some real numbers with $A>0$ and $B>1$. Furthermore, let
		\[
        \epsilon:=||\mu q|| - M \cdot ||\tau q||.
		\]
		If $\epsilon >0$, then there is no solution to the inequality 
		\[
        0< |u \tau - v + \mu| <AB^{-w} 
        \]
		in positive integers $u$, $v$ and $w$ with
		\[
		u \leq M \quad\text{and}\quad w \geq \frac{\log(Aq/\epsilon)}{\log B}.
		\]	
	\end{lemma}
Note that Lemma \ref{lem 2.2} cannot be applied when $\mu = 0$ (since then $\epsilon < 0$) or when $\mu$
is a multiple of $\tau$. For this case, we use the following well-known technical result from Diophantine approximation, known as Legendre’s criterion.
\begin{lemma}\label{lem 2.3}
Let $\tau$ be a real number and $r, s$ integers such that
\[
\left|\tau - \frac{r}{s}\right| < \frac{1}{2 s^{2}}.
\]
Then $r/s = p_{k}/q_{k}$ is a convergent of the continued fraction expansion $[a_{0}, a_{1},\ldots]$ of $\tau$ (with
some $k = 0, 1,\ldots$). Further, let $M$ and $N$ be nonnegative integers such that $q_N > M$. Then putting $a(M) := \max \{a_i: i=0,1,2,\dots, N \}$, the inequality
		\[
		\Bigm| \tau - \frac{r}{s} \Bigm| > \frac{1}{(a(M)+2)s^2},
		\]
		holds for all pairs $(r, s)$ of positive integers with $0 < s < M$.
	\end{lemma}

\subsection{Properties of Padovan and Perrin sequences}\label{subsec2.3}
	In this subsection, we begin by recalling some properties of
	these ternary recurrence sequences that will be used later.
	The characteristic equation of these sequences are $x^3-x-1=0$ which has one real root $\alpha$ and two complex conjugate roots $\beta$ and $\gamma = \overline{\beta}$, where
	\[
	\alpha=\dfrac{r_1+r_2}{6},\quad \beta=\dfrac{-r_1-r_2+i\sqrt{3}(r_1-r_2)}{12},
	\]
	with $r_{1}=\sqrt[3]{108 + 12\sqrt{69}}$ and $r_{2}=\sqrt[3]{108-12\sqrt{69}}.$  The Binet's formula of $\mathcal{P}_n$ is
	\begin{equation}\label{eq 2.4}
		\mathcal{P}_{n}=p \alpha^n + q \beta^n + r \gamma^n,\; \text{ for all } n\geq 0,
	\end{equation}
	and that for $E_{n}$ is
	\begin{equation}\label{eq 2.5}
		E_n=\alpha^n+\beta^n+\gamma^n,\; \text{ for all } n\geq 0,
	\end{equation}
	where
	\begin{align}\label{eq 2.6}
		p &= \dfrac{(1-\beta)(1-\gamma)}{(\alpha-\beta)(\alpha-\gamma)}= \dfrac{1+\alpha}{-\alpha^2+3\alpha+1},\vspace{1mm}\nonumber \\
		q &= \dfrac{(1-\alpha)(1-\gamma)}{(\beta-\alpha)(\beta-\gamma)}=\dfrac{1+\beta}{-\beta^2+3\beta+1},\vspace{1mm}\\
		r &= \dfrac{(1-\alpha)(1-\beta)}{(\gamma-\alpha)(\gamma -\beta)}=\dfrac{1+\gamma}{-\gamma^2+3\gamma+1}=\overline{q}.\nonumber
	\end{align}
	The minimal polynomial of $p$ over $\mathbb{Z}$ is  $23x^{3} - 23x^{2} + 6x - 1$ with all its zeros lie strictly inside the unit circle. Numerically, we have the following estimates:
	\begin{align}\label{eq 2.7}
		1.32<&\alpha<1.33,\nonumber \\
		0.86<|\beta|=&|\gamma| =\alpha^{-1/2} <0.87,\nonumber \\
		0.72<&p<0.73,\\
		0.24<|q|&=|r|<0.25.\nonumber
	\end{align}
	In particular, setting
	\begin{equation}\label{eq 2.8}
		\pi(n):= \mathcal{P}_{n} - p\alpha^n = q\beta^n + r\gamma^n, \quad \text{we have} \quad |\pi(n)|< \frac{1}{\alpha^{n/2}}
	\end{equation}
	and
	\begin{equation}\label{eq 2.9}
		\psi(n):= E_{n} - \alpha^n = \beta^n + \gamma^n, \quad \text{we have} \quad |\psi(n)|< \frac{2}{\alpha^{n/2}}
	\end{equation}
	holds for all $n \geq 1$. Further, using induction, one can prove that
	\begin{equation}\label{eq 2.10}
		\alpha^{n-3}\leq \mathcal{P}_n \leq \alpha^{n-1}, \quad \text{for all} \quad n \geq 1
	\end{equation}
	and 
	\begin{equation}\label{eq 2.11}
		\alpha^{n-2}\leq E_n \leq \alpha^{n+1},\quad \text{for all} \quad n\geq 2.
	\end{equation}
	
	\subsection{Properties of Narayana's cows sequence}
	The characteristic polynomial of Narayana's cows sequence is 
	$f(x) = x^{3} - x^{2} - 1.$ This polynomial is irreducible in $\mathbb{Q} [x]$ and has only real root $\varphi$ which has absolute value greater than $1$, while the other two conjugate complex roots are $\lambda$ and $\delta = \Bar{\lambda}$ with $|\lambda| = |\delta|< 1$. The Binet formula for $N_{n}$ is given by 
	\begin{equation}\label{eq 2.12}
		N_{n} = C_{\varphi} \varphi^{n} + C_{\lambda} \lambda^{n} + C_{\delta} \delta^{n},\; \text{ for all } n\geq 0,
	\end{equation}
	where
	\begin{align}\label{eq 2.13}
		C_{\varphi} = \frac{\varphi}{(\varphi-\lambda)(\varphi-\delta)}, \quad \quad
		C_{\lambda} = \frac{\lambda}{(\lambda-\varphi)(\lambda-\delta)}, \quad \quad
		C_{\delta} = \frac{\delta}{(\delta-\varphi)(\delta-\lambda)}.
	\end{align}
	The coefficient $C_{\varphi}$ has the minimal polynomial $31x^{3} - 3x - 1$  over $\mathbb{Z}$ and
	all the zeros of this polynomial lie strictly inside the unit circle. Numerically, we have the following:
		\begin{align}\label{eq 2.14}
		1.46 < & \varphi < 1.47,\nonumber \\
		0.82 <|\lambda| &= |\delta| < 0.83,\nonumber \\
		0.41 < & C_{\varphi} < 0.42,\\
		0.27 < |C_{\lambda}|& = |C_{\delta}| < 0.28.\nonumber
	\end{align}
	Considering the estimates for $\varphi, \lambda, \delta, C_{\varphi}, C_{\lambda}, C_{\delta}$ and putting
	\begin{equation}\label{eq 2.15}
		e(n) := N_{n} - C_{\varphi} \varphi^{n} = C_{\lambda} \lambda^{n} + C_{\delta} \delta^{n}, \quad \text{we have that} \quad |e(n)| < \frac{1}{\varphi^{n/2}},
	\end{equation}
	for all $n \geq 1$. Using induction, it can be seen that 
	\begin{equation}\label{eq 2.16}
		\varphi^{n-2} \leq N_{n} \leq \varphi^{n-1} \quad \text{hold for all} \quad n \geq 1.
	\end{equation}
\subsection{Other useful lemmas}
		We conclude this section by recalling two lemmas that we will need in this work.
	
	\begin{lemma}\label{lem 2.4} {\rm{(\cite{Bravo}, Lemma 8)}}
For any non-zero real number $x$, we have

(a) $0 < x < |e^{x} - 1|$.

(b) If $x < 0$ and $|e^{x} - 1| < 1/2$, then $|x| < 2 |e^{x} - 1|$.
	\end{lemma}
	
	\begin{lemma}\label{lem 2.5} {\rm{(\cite{Sanchez}, Lemma 7)}}
		If $m \geq 1$, $S \geq (4m^{2})^{m}$ and $\frac{x}{(\log x)^{m}} < S$, then $x < 2^{m} S (\log S)^{m}$.
	\end{lemma}

\section{Thabit and Williams numbers base $b$ which are Padovan numbers}
In this section, we will prove Theorem \ref{thm1} and then give an illustration of this result for $2 \leq b \leq 10$.
\subsection{Proof of Theorem \ref{thm1}}
We start by assuming $n > 300$. Combining estimates \eqref{eq 2.10} together with the equation \eqref{eq 1.1} yields
\[
(b-1)b^l-1\leq (b\pm 1)b^l\pm 1 = P_n \leq \alpha^{n-1},
\]
which implies that
\begin{equation}\label{eq 3.12}
b^l < (b -1)b^l \leq \alpha^{n-1} + 1 < \alpha^{n+1}.
\end{equation}
Taking logarithms on both sides of inequalities \eqref{eq 3.12} ,
we obtain
\begin{equation}\label{eq 3.13}
l< n \quad \text{holds for all} \quad n > 300 \quad \text{and} \quad  b \geq 2.
\end{equation}
Using Binet's formula \eqref{eq 2.8}, we write the equation \eqref{eq 1.1} as
\[
(b\pm 1)b^l\pm 1=p \alpha^n + \pi(n)
\]
which further implies
\[
(b\pm 1)b^l - p \alpha^n = \pi(n) \pm 1. 
\]
Taking absolute values on both sides, we obtain
\[
\left|(b\pm 1)b^l - p\alpha^n \right|< \left| \pi(n) \right| + 1.
\]
Dividing both sides of the above inequality by $p\alpha^n$, we get
\begin{equation}\label{eq 3.14}
|\Lambda_{1}|<\frac{2.77}{\alpha^n},
\end{equation}
where
\begin{equation}\label{eq 3.15}
\Lambda_{1}=(b\pm 1)b^l p^{-1 }\alpha^{-n}-1.
\end{equation}
To apply Theorem \ref{thm4}, we need to show that $\Lambda_{1} \neq 0$. Indeed, $\Lambda_{1}= 0$ implies
\[
(b\pm 1)b^l=p\alpha^n.
\]
If we conjugate the above relation with an automorphism $\sigma$ of the Galois extension of $\mathbb{Q}(\alpha )$ over $\mathbb{Q}$ given by $\sigma(\alpha)= \beta$ and then taking absolute values, we obtain
	\[
	(b\pm 1)b^l = |q| |\beta|^{n}
	\]
which leads us to a contradiction as left-hand side exceeds $1$ for all $b \geq 2$, while its right-hand side is less than than 1. Hence, $\Lambda_{1} \neq 0$. Thus, we apply Theorem \ref{thm4} to $\Lambda_{1}$ given by \eqref{eq 3.15} with the parameters:
\[
\eta_{1} := (b\pm 1)p^{-1 }, \quad \eta_{2} :=b  , \quad \eta_{3} := \alpha,
\]
and
\[ b_{1}:= 1, \quad b_{2}:= l, \quad b_{3}:= -n.
\]
Note that the three numbers $\eta_{1}, \eta_{2}, \eta_{3}$ are positive real numbers and belong to the field  $\mathbb{L} := \mathbb{Q}(\alpha)$, so we can take $d_{\mathbb{L}} = [\mathbb{L}:\mathbb{Q}] = 3$. Since $h(\eta_{2}) = \log b $ and $h(\eta_{3}) =\frac{\log \alpha}{3}$, it follows that  
\[
\max\{3h(\eta_{2}),|\log \eta_{2}|,0.16\} = 3\log b := B_{2}
\]
and 
\[
\max\{3h(\eta_{3}),|\log \eta_{3}|,0.16\} = \log \alpha := B_{3}.
\]
On the other hand, by using the  properties of logarithmic height, it follows that 
\[
 h(\eta_{1}) \leq  h\left((b\pm 1)p^{-1 }\right) < h\left(b\pm 1\right) + h(p) < \log b + \log 2+ \frac{\log 23}{3}.
\]
Thus, we obtain
\[
\max\{3h(\eta_{1}),|\log \eta_{1}|,0.16\} = 3\log b+ 3 \log 2 + \log 23 := B_{1}.
\]
We can take $D:= \max \{|1|, |l|, |-n| \} = n$ by \eqref{eq 3.13}. Applying Matveev's theorem we get
\begin{equation}\label{eq 3.16}
- \log |\Lambda_{1}| < 1.4 \times 30^{6} \times 3^{4.5} \times (3)^{2} (1+ \log 3) (1 +\log n) (3\log b+ 3 \log 2  + \log 23 )  (3\log b) (\log \alpha).    
\end{equation}
From the comparison of lower bound \eqref{eq 3.16} and upper bound \eqref{eq 3.14} of $|\Lambda_{1}|$ gives us
\begin{equation}\label{eq 3.17}
n\log \alpha- \log 2.77< 2.74\times 10^{12} \log n \log b \log(184b^3)
\end{equation}
where we have used the fact $1+ \log n < 1.2 \log n$ for all  $n > 300$. The last inequality  \eqref{eq 3.17}, becomes
\[
\frac{n}{\log n} < 9.78 \times 10^{12} \log b \log (184 b^{3}).
\]
Now, we apply Lemma \ref{lem 2.5} with $S := 9.78 \times 10^{12}\log b \log (184b^3$), $x = n$ and $m = 1$. So,
we have 
\[
n < 1.96\times 10^{13}\log b \log (184 b^3) (29.91+\log(\log b)+\log (\log (184b^3))).
		\] 
This completes the proof of Theorem \ref{thm1}.
\subsection{Application for $2 \leq b \leq 10$}
In this subsection, we will determine all solutions to the Diophantine equation \eqref{eq 1.1} for $l \geq 1$ and
$2 \leq b \leq 10$.  So in this range we have the following result.
\begin{theorem}\label{thm5} 
Let $b$ be a positive integer such that $2 \leq b \leq 10$. 
\begin{enumerate}
    \item  Then, the solutions $(n, b, l)$ to the Diophantine equation $\mathcal{P}_{n} = (b+1)b^l-1$
is $\{(7,2,1)\}$. More precisely, the Thabit numbers base $b$ of first kind which are Padovan numbers is $5$.

    \item  Then, the solutions $(n, b, l)$ to the Diophantine equation $\mathcal{P}_{n} = (b+1)b^l+1$
are in
\[
\left\{(8, 2, 1), (12, 4, 1), (14, 3, 2), (15, 2, 4), (19, 5, 2)\right\}.
\] 
More precisely, the Thabit numbers base $b$ of second kind which are Padovan numbers are $7$, $21$, $37$, $49$ and $151$.

\item Then, the solutions $(n, b, l)$ to the Diophantine equation $\mathcal{P}_{n} = (b-1)b^l-1$ are in
\[
\left\{ (0, 2, 1), (1, 2, 1), (2, 2, 1), (5, 2, 2), (7, 3, 1), (8, 2, 3)\right\}.
\]
More precisely, the Williams numbers base $b$ of first kind which are Padovan numbers are $1$, $3$, $5$, $7$.

\item  Then, the solutions $(n, b, l)$ to the Diophantine equation $\mathcal{P}_{n} = (b-1)b^l+1$ are in
\[
\left\{(5, 2, 1), (7, 2, 2), (8, 3, 1), (9, 2, 3), (12, 5, 1), (15, 4, 2), (16, 2, 6), (26, 6, 3)\right\}.
\]
More precisely, the Williams numbers base $b$ of second kind which are Padovan numbers are $3$, $5$, $7$, $9$, $21$, $49$, $65$, and $1081$.
\end{enumerate}    
\end{theorem}
\begin{proof}
Since $2 \leq b \leq 10$, from Theorem \ref{thm1} we obtain $n < 1.82 \times 10^{16}$. To reduce these bounds we need to apply Lemma \ref{lem 2.2}. Let us define
	\begin{equation}\label{eq 3.19}
		\Gamma_{1} := l\log b-n\log\alpha+\log (b\pm 1) p^{-1}.  
	\end{equation}
	Then $e^{\Gamma_{1}} - 1 := \Lambda_{1}$, where $\Lambda_{1}$ is defined by \eqref{eq 3.15}. Therefore, \eqref{eq 3.14} implies that
	\[
    |e^{\Gamma_{1}} -1 | < \frac{2.77}{\alpha^n}.
    \]
	Note that $\Gamma_{1} \neq 0$. Thus, we distinguish the following cases. If $\Gamma_{1} > 0$, then we can apply Lemma \ref{lem 2.4} (a) to obtain
	\[
	0 < \Gamma_{1} < e^{\Gamma_{1}} - 1 <  \frac{2.77}{\alpha^n}.
	\]
	If $\Gamma_{1} < 0$, then from \eqref{eq 3.14} we have that $|e^{\Gamma_{1}} - 1|< 1/2$ and therefore $e^{|\Gamma_{1}|} < 2$. Thus, by Lemma \ref{lem 2.4} (b), we have that 
	\[
	0 < |\Gamma_{1}| \leq e^{|\Gamma_{1}|} - 1 = e^{|\Gamma_{1}|}|e^{\Gamma_{1}} - 1 | <\frac{5.54}{\alpha^n}.
	\]
	So in both cases, we have 
	\begin{equation}\label{eq 3.20}
		0 < |\Gamma_{1}| <  5.54\cdot \alpha^{-n}. 
	\end{equation}
	Inserting \eqref{eq 3.19} in \eqref{eq 3.20} and dividing both sides by $\log \alpha$, it results that
	\begin{equation}\label{eq 3.21}
		0<\left| l \left(\frac{\log b}{\log \alpha} \right)-n+ \frac{\log (b\pm 1)p^{-1}}{\log \alpha}\right|<{19.7} \cdot {\alpha^{-n}}.  	
	\end{equation}
    In order to apply Lemma \ref{lem 2.2}, we set
	\[ \tau := \frac{\log b}{\log \alpha},  \quad \mu := \frac{\log (b\pm 1)p^{-1}}{\log \alpha}, \quad A:= 19.7 \quad \text{and} \quad B:= \alpha. 
	\]
Note that  $\tau =  \frac{\log b}{\log \alpha} \notin \mathbb{Q}$ because if $\frac{\log b}{\log \alpha} = \frac{s}{t}$ for some coprime positive integers $s$ and $t$, we would have $b^{t} = \alpha^{s}  \in  \mathbb{Z}$. Using the automorphism $\sigma \in Gal(\mathbb{L} / \mathbb{Q})$, the Galois group of the extension $\mathbb{L} / \mathbb{Q}$, such that $\sigma(\alpha)=\beta$. Applying this to the above relation and taking absolute values we get $1 < b^{t} = |\beta|^{s} < 1$,  which is a contradiction. Now, we apply Lemma \ref{lem 2.2} with $M := 1.82 \times 10^{16}$ which is the upper bound of $n$. For every $b$ such that $2 \leq b \leq 10$, we found that $q_{44}$ the denominator of the $44$-th convergent of $\tau$ exceeds $6M$ and $\epsilon(b) :=||\mu q_{44}|| - M ||\tau q_{44}|| > 0$. Then, with the help of \textit{Mathematica}, we can say that if the inequality \eqref{eq 3.21} has a solution, then
\[
	n < \frac{\log ( A q_{44}/\epsilon)}{\log B} < 212,
	\]
 contradicting the fact that $n > 300$. Now, we search for
the solutions to the Diophantine equation \eqref{eq 1.1} with 
\[
1 \leq l < n, \quad 0 \leq n \leq 300 \quad \text{and} \quad 2 \leq b \leq 10. 
\]
Using \textit{Mathematica}, we checked that all the solutions of the Diophantine equation \eqref{eq 1.1} are those listed in the statement of Theorem \ref{thm5}. This completes the proof.
\end{proof}

\section{Thabit and Williams numbers base $b$ which are Perrin numbers}
In this section, we will prove Theorem \ref{thm2} and then give an illustration of this result for $2 \leq b \leq 10$.
\subsection{Proof of Theorem \ref{thm2}}
We begin by considering $n > 350$. Using the inequality \eqref{eq 2.11} and equation \eqref{eq 1.2}, we obtain
\[
(b-1)b^l-1\leq (b\pm 1)b^l\pm 1 = E_n < \alpha^{n+1}
\]
leads to
\begin{equation}\label{eq 4.22}
b^l<(b -1)b^l\leq \alpha^{n+1}+1< \alpha^{n+2}.
\end{equation}
Taking logarithms on both sides of inequalities \eqref{eq 4.22} , we obtain
\begin{equation}\label{eq 4.23}
l< n \quad \text{for all} \quad n > 350,\quad  b\geq 2.
\end{equation}
By employing Binet's formula \eqref{eq 2.9}, equation \eqref{eq 1.2} can be expressed as
\[
(b\pm 1)b^l\pm 1= \alpha^n + \psi(n)
\]
which we rewrite as
\[
(b\pm 1)b^l -  \alpha^n = \psi(n) \pm 1. 
\]
Taking absolute values on both sides, we obtain
\[
\left|\alpha^n-(b\pm 1)b^l \right|< \left| \psi(n) \right| + 1.
\]
Dividing both sides of the above inequality by $\alpha^n$, we get
\[
\left|(b\pm 1)b^l \alpha^{-n}-1\right| < \frac{3}{\alpha^n}.
\]
We deduce that
\begin{equation}\label{eq 4.25}
|\Lambda_{2}| < \frac{3}{\alpha^n}, \quad \text{where} \quad \Lambda_{2}=(b\pm 1)b^l \alpha^{-n}-1.
\end{equation}
We can prove that $\Lambda_{2} \neq 0$ by a similar method used to show that $\Lambda_{1} \neq 0$. Now, let us apply Theorem \ref{thm4} to $\Lambda_{2}$ given by \eqref{eq 4.25} with the following parameters
		\[
		s:=3, \quad \quad (\eta_{1}, a_{1}) := \left(b\pm 1, 1 \right), \quad \quad (\eta_{2}, a_{2}) := \left( b, l \right),~ \quad {\rm and} ~  \quad (\eta_{3}, a_{3}) := \left(\alpha, -n \right).
		\]
The number field containing $\eta_{1}, \eta_{2}, \eta_{3}$ is  $\mathbb{L} := \mathbb{Q}(\alpha)$, which has degree $d_{\mathbb{L}} = [\mathbb{L}:\mathbb{Q}]= 3$. As calculated before, we take $B_{2} = 3 \log b $ and $B_{3} =\log \alpha $. On the other hand, by using the  properties of logarithmic height, it follows that 
\[
 h(\eta_{1}) \leq  h\left(b\pm 1 \right) < \log b+ \log 2 .
\]
Thus, we obtain
\[
\max\{3 h(\eta_{1}),|\log \eta_{1}|,0.16\} = 3\log b+ 3\log 2 := B_{1}.
\]
In addition by \eqref{eq 4.23}, we can take $D:= \max \{|1|, |l|, |-n| \} = n$. Therefore, according to Theorem \ref{thm4}, we have
\begin{equation}\label{eq 4.26}
- \log |\Lambda_{2}| < 1.4 \times 30^{6} \times 3^{4.5} \times 3^{2} (1+ \log 3) (1 +\log n) (3\log b+ 3\log 2 )  (3\log b) (\log \alpha).    
\end{equation}
 Comparing this lower bound \eqref{eq 4.26} with
the upper bound of $|\Lambda_{2}|$ provided by \eqref{eq 4.25}, we arrive at
\begin{equation}\label{eq 4.27}
n\log \alpha- \log 3 < 2.74\times 10^{12} \log n \log b \log(8b^3)
\end{equation}
where we used the fact that inequality $1 + \log n < 1.2 \log n$ holds for $n > 350$. Consequently, using \eqref{eq 4.27} yields
\[
 \frac{n}{\log n} < 9.78\times 10^{12} \log b \log (8b^{3}).
\]
Applying the inequality from Lemma \ref{lem 2.5} with $S := 9.78 \times 10^{12} \log b \log (8b^{3})$, $x = n$ and $m = 1$, we obtain
\[
n < 1.96 \times 10^{13}\log b (\log (8b^3))(29.91+\log(\log b)+\log (\log (8b^3))).
		\] 
This establishes and finishes the proof of Theorem \ref{thm2}.
\subsection{Application for $2 \leq b \leq 10$}
In this part, we will search for all solutions to the Diophantine equation \eqref{eq 1.2} for $l \geq 1$ and
$2 \leq b \leq 10$.  Thus, we get the following result.
\begin{theorem}\label{thm6}
Let $b$ be a positive integer such that $2 \leq b \leq 10$. 
\begin{enumerate}
    \item Then, the solutions $(n, b, l)$ to the Diophantine equation $E_{n} = (b+1)b^l-1$ are in
\[
\left\{(5,2,1), (6,2,1), (12,5,1)\right\}. 
\]
More precisely, the Thabit numbers base $b$ of first kind which are Perrin numbers are $5$ and $29$.
\item Then, the solutions $(n, b, l)$ to the Diophantine equation $E_{n} = (b+1)b^l+1$ is $\left\{ (7,2,1) \right\}$. 
More precisely, the Thabit numbers base $b$ of second kind which are Perrin is $7$.
\item Then, the solutions $(n, b, l)$ to the Diophantine equation $E_{n} = (b-1)b^l-1$ are in
\[
\left\{(0,2,2), (3,2,2), (5,3,1), (6,3,1), (7,2,3), (10,3,2), (12,6,1)\right\}.
\] 
More precisely, the Williams numbers base $b$ of first kind which are Perrin numbers are  $3$, $5$, $7$, $17$ and $29$.

\item Then, the solutions $(n, b, l)$ to the Diophantine equation $E_{n} = (b-1)b^l+1$ are in
\[
\left\{(0,2,1), (3,2,1), (5,2,2), (6,2,2), (7,3,1), (10,2,4)\right\}.
\]
More precisely, the Williams numbers base $b$ of second kind which are Perrin numbers are $3$, $5$, $7$ and $17$.
\end{enumerate}  
\end{theorem}
\begin{proof}
Since $2 \leq b \leq 10$, from Theorem \ref{thm2} we obtain $n < 1.34 \times 10^{16}$. To reduce these bounds we need to apply Lemma \ref{lem 2.2}. Let us define
	\begin{equation}\label{eq 4.29}
		\Gamma_{2} := l \log b - n \log \alpha +\log (b\pm 1).  
	\end{equation}
Therefore, inequality \eqref{eq 4.25} can be rewritten as $|e^{\Gamma_{2}} -1 | < 3/ \alpha^{n}.$ In this case, $\Gamma_{2} \neq 0$. If $\Gamma_{2} > 0$, then we can apply Lemma \ref{lem 2.4} (a) to obtain $|\Gamma_{2}| < 3/\alpha^{n}$. Now, if $\Gamma_{2} < 0$, then  $|e^{\Gamma_{2}} - 1|< 1/2$. Thus, by Lemma \ref{lem 2.4} (b), we have that $|\Gamma_{2}| < 2 |e^{\Gamma_{2}} - 1 | < 6/\alpha^{n}$. So, in both cases we have
	\begin{equation}\label{eq 4.30}
		0 < |\Gamma_{2}| <  6\cdot \alpha^{-n}. 
	\end{equation}
Replacing \eqref{eq 4.30} by \eqref{eq 4.29} and dividing through by $\log \alpha$, we obtain
	\begin{equation}\label{eq 4.31}
		0<\left| l \tau -n + \mu \right|< A \cdot B^{-n}.  	
	\end{equation}
where	
	\[ \tau := \frac{\log b}{\log \alpha},  \quad \mu :=  \frac{\log (b\pm 1)}{\log \alpha}, \quad A:= 21.34 \quad \text{and} \quad B:= \alpha. 
	\]	
Now, we apply Lemma \ref{lem 2.2} to \eqref{eq 4.31} by taking $M := 1.34 \times 10^{16}$ which is an upper bound on $n$. For every $b$ such that $2 \leq b \leq 10$ we found that $q_{43}$ the denominator of the $43$-th convergent of $\tau$ exceeds $6M$. A quick computation with \textit{Mathematica} gives us that the inequality $\epsilon(b) :=||\mu q_{43}|| - M ||\tau q_{43}|| > 0$  except for the case $b=2$. Then, with the help of Mathematica, we can say that if the inequality \eqref{eq 4.31} has a solution for $3 \leq b \leq 10$, then 
\[
	n < \frac{\log ( A q_{43}/\epsilon)}{\log B} < 219.
	\]
On the other hand, in the case $b =2$; from \eqref{eq 4.31}, we have 
\[
0<\left| l \left(\frac{\log b}{\log \alpha} \right) - n \right| < 21.34 \cdot {\alpha^{-n}}. 
\]
If we divide this inequality by $l$, we get
\[
0<\left| \frac{\log b}{\log \alpha} - \frac{n}{l} \right| < \frac{21.34}{l \cdot \alpha^{n}}.
\]
For $n > 350$, it can be seen that 
\[
\frac{\alpha^{n}}{2(21.34)} > 1.7 \times 10^{41} > 1.34 \times 10^{16} > l,
\]
and so we have
\[
\left| \frac{\log b}{\log \alpha} - \frac{n}{l} \right| < \frac{21.34}{l \cdot \alpha^{n}} < \frac{1}{2l^{2}}.
\]
Therefore, from Lemma \ref{lem 2.3}, we conclude that the rational number $\frac{n}{l}$ is a convergent to $\tau = \frac{\log b}{\log \alpha}$ with $b=2$. Let 
		\[
		[a_{0}, a_{1}, a_{2},\dots] := [2, 2, 6, 1, 1, 1, 2, 1, 13, 3, 1, 1, 1, 1, 1, 8, 1, 3, 2, 2, 7, 1, 2, \
5, 1, 2 \dots]
		\]
		be the continued fraction expansion of $\tau$. Since $l < 1.34 \times 10^{16}$, we can apply Lemma \ref{lem 2.3} with $M := 1.34 \times 10^{16}$. A quick search using \textit{Mathematica} reveals that 
		\[
		q_{42} < 1.34 \times 10^{16} < q_{43}.
		\]
		Furthermore, $a(M):= \max\{a_{i} : i=0, 1, \dots,43\} = 80$. So by Lemma \ref{lem 2.3}, we have
		\[
        \frac{1}{(a(M)+2) \cdot l^{2}} \leq \left|\frac{\log b}{\log \alpha} - \frac{n}{l} \right| < \frac{21.34}{l \cdot \alpha^{n}}.
        \]
This yields to
\[
n < \frac{\log (21.34 \cdot (a(M)+2) \cdot 1.34 \times 10^{16})}{\log \alpha} < 159.
\]
Therefore, in all cases, we have $n < 219$. This is a contradiction to the fact that $n > 350$. Now, we search for the solutions to the Diophantine equation \eqref{eq 1.2} with 
\[
1 \leq l < n, \quad 0 \leq n \leq 350 \quad \text{and} \quad 2 \leq b \leq 10. 
\]
Using \textit{Mathematica}, we checked that all the solutions of the Diophantine equation \eqref{eq 1.2} are those listed in the statement of Theorem \ref{thm6}. This completes the proof.
\end{proof}

\section{Thabit and Williams numbers base $b$ which are Narayana's cows numbers}
In this section, we will prove Theorem \ref{thm3} and then give an illustration of this result for $2 \leq b \leq 10$.
\subsection{Proof of Theorem \ref{thm3}}
We start by assuming $n > 400$. By inequality \eqref{eq 2.16}, the equation \eqref{eq 1.3} implies that
\[
(b-1)b^l-1\leq (b\pm 1)b^l\pm 1 = N_n< \varphi^{n-1}
\]
which gives
\begin{equation}\label{eq 5.34}
b^l<(b -1)b^l\leq \varphi^{n-1}+1 < \varphi^{n+1}.
\end{equation}
Taking logarithms on both sides of inequalities \eqref{eq 5.34}, we obtain
\begin{equation}\label{eq 5.35}
l< n \quad \text{holds for all} \quad n > 400\quad \text{and} \quad b\geq 2.
\end{equation}
Using Binet's formula \eqref{eq 2.15}, equation \eqref{eq 1.3} can be rearranged as
\[
(b \pm 1) b^l \pm 1 = C_{\varphi} \varphi^{n} + e(n) 
\]
and consequently we have
\[
 (b\pm 1)b^l - C_{\varphi} \varphi^{n} =  e(n) \pm 1. 
\]
Dividing by $C_{\varphi} \varphi^{n}$ and taking absolute values, we obtain
\begin{equation}\label{eq 5.36}
|\Lambda_{3}| < \frac{3.4}{\varphi^n} \quad \text{where} \quad \Lambda_{3}=(b\pm 1)b^l C_{\varphi}^{-1 }\varphi^{-n}-1.
\end{equation}
We have $\Lambda_{3} \neq 0$, otherwise we would get
\[
(b\pm 1)b^l=C_{\varphi}\varphi^{n}
\]
If we conjugate the above relation with an automorphism $\sigma$ of the Galois extension of $\mathbb{Q}(\varphi )$ over $\mathbb{Q}$ given by $\sigma(\varphi)= \lambda$ and then taking absolute values, we obtain
	\[
	(b\pm 1)b^l = |C_{\lambda}| |\lambda|^{n}
	\]
which leads us to a contradiction as left-hand side exceeds $1$ for all $b \geq 2$, while its right-hand side is less than than 1. Thus, $\Lambda_{3} \neq 0$. Thus, we apply Theorem \ref{thm4} to $\Lambda_{3}$ given by \eqref{eq 5.36} with the parameters:
\[
		s:=3, \quad \quad (\eta_{1}, a_{1}) := \left((b\pm 1)C_{\varphi}^{-1 }, 1 \right), \quad \quad (\eta_{2}, a_{2}) := \left( b, l \right),~ \quad {\rm and} ~  \quad (\eta_{3}, a_{3}) := \left(\varphi, -n \right).
		\]
Then, we obtain
\[
h(\eta_{1}) \leq  h\left((b\pm 1)C_{\varphi}^{-1 }\right) < \log b+ \log 2+ \frac{\log 31}{3}, \quad h(\eta_{2}) = \log b \quad \text{and} \quad h(\eta_{3}) =\frac{\log  \varphi}{3}.
\]
It is clear that $\mathbb{L} := \mathbb{Q}(\varphi)$ contains $\eta_{1}, \eta_{2}, \eta_{3}$ and has degree $d_{\mathbb{L}} = [\mathbb{L}:\mathbb{Q}] = 3$. So, we take
\[
B_{1} := 3\log b+ 3\log 2+ \log 31, \quad B_{2} := 3\log b, \quad B_{3} := \log \varphi \quad \text{and} \quad D:= n.
\]
Applying Theorem \ref{thm4} and comparing the resulting inequality with \eqref{eq 5.36}, we obtain
\begin{equation}\label{eq 5.37}
n\log \varphi- \log 3.4 < 3.72 \times 10^{12} \log n \log b \log(248b^3)
\end{equation}
where we used the fact that inequality $1 + \log n < 1.2 \log n$ holds for $n > 400$. So, inequality \eqref{eq 5.37} becomes
\[
\frac{n}{\log n} < 9.76 \times 10^{12} \log b \log (248b^{3}).
\]
Thus, putting $S := 9.76 \times 10^{12} \log b \log (248b^3)$, $x = n$ and $m = 1$ in Lemma \ref{lem 2.5}, we obtain 
\[
n < 1.95 \times 10^{13} \log b \log (248 b^3) (29.91 + \log(\log b) + \log (\log 248b^3)).
\] 
This establishes and completes the proof of Theorem \ref{thm3}.
\subsection{Application for $2 \leq b \leq 10$}
In this subsection, we give all the solutions of the Diophantine equation \eqref{eq 1.3} for $l \geq 1$ and
$2 \leq b \leq 10$. Here is our result in this range. 
\begin{theorem}\label{thm7}
Let $b$ be a positive integer such that $2 \leq b \leq 10$. 
\begin{enumerate}
    \item Then, the solutions $(n, b, l)$ to the Diophantine equation $N_{n} = (b+1)b^l-1$ are in
    \[
    \left\{(9,4,1), (11,6,1)\right\}.
    \]
    More precisely, the Thabit numbers base $b$ of first kind which are Narayana's cows numbers are $19$ and $41$.
    \item Then, the solutions $(n, b, l)$ to the Diophantine equation $N_{n} = (b+1)b^l + 1$ are in
    \[
    \left\{(8,2,2), (8,3,1), (22,7,3)\right\}.
    \]
More precisely, the Thabit numbers base $b$ of second kind which are Narayana's cows numbers are $13$ and $2745$.
\item Then, the solutions $(n, b, l)$ to the Diophantine equation
$N_{n} = (b-1)b^l-1$ are in
\[
\left\{(0,2,1), (1,2,1), (2,2,1), (4,2,2), (9,5,1), (11,7,1)\right\}.
\]
More precisely, the Williams numbers base $b$ of second kind which are Narayana's cows numbers are $1$, $3$, $19$ and $41$.
\item Then, the solutions $(n, b, l)$ to the Diophantine equation
$N_{n} = (b-1)b^l+1$ are 
\[
\left\{(4,2,1), (7,2,3), (8,4,1), (9,3,2), (14,2,7)\right\}.
\]
More precisely, the Williams numbers base $b$ of second kind which are Narayana's cows numbers are $3$, $9$, $13$, $19$ and $129$.
\end{enumerate}
\end{theorem}
\begin{proof}
Since $2 \leq b \leq 10$, from Theorem \ref{thm3} we obtain $n < 1.85 \times 10^{16}$. To reduce these bounds we need to apply Lemma \ref{lem 2.2}. Let us define
	\begin{equation}\label{eq 5.39}
		\Gamma_{3} := l\log b- n \log\varphi+\log (b\pm 1) C_{\varphi}^{-1}.  
	\end{equation}
From the inequality \eqref{eq 5.36}, one can see that
	\begin{equation}\label{eq 5.40}
		0 < |\Gamma_{3}| <  6.8 \cdot \varphi^{-n}. 
	\end{equation}
	Inserting \eqref{eq 5.39} in \eqref{eq 5.40} and dividing both sides by $\log \varphi$, we get
	\begin{equation}\label{eq 5.41}
		0<\left| l\left(\frac{\log b}{\log \varphi} \right)-n + \frac{\log (b\pm 1)C_{\varphi}^{-1}}{\log \varphi}\right|<{17.79} \cdot {\varphi^{-n}}.  	
	\end{equation}
    In order to apply Lemma \ref{lem 2.2}, we set
	\[ \tau := \frac{\log b}{\log \varphi},  \quad \mu := \frac{\log (b\pm 1)C_{\varphi}^{-1}}{\log \varphi}, \quad A:= 17.79 \quad \text{and} \quad B:= \varphi. 
	\]
Note that  $\tau =  \frac{\log b}{\log \varphi} \notin \mathbb{Q}$. Now, we apply Lemma \ref{lem 2.2} with $M := 1.85\times 10^{16}$ which is the upper bound of $n$. 
For $2 \leq b \leq 10$,  we found that $q_{44}$ the denominator of the $44$-th convergent of $\tau$ exceeds $6M$ and $\epsilon(b) :=||\mu q_{44}|| - M ||\tau q_{44}|| > 0$. Then, with the help of \textit{Mathematica}, we can say that if the inequality \eqref{eq 5.41} has a solution, then
\[
	n < \frac{\log ( A q_{44}/\epsilon)}{\log B} < 169,
	\]
contradicting the fact that $n > 400$. Now, we search for
the solutions to the Diophantine equation \eqref{eq 1.3} with 
\[
1 \leq l < n, \quad 0 \leq n \leq 400 \quad \text{and} \quad 2 \leq b \leq 10. 
\]
Using \textit{Mathematica}, we checked that all the solutions of the Diophantine equation \eqref{eq 1.3} are those listed in the statement of Theorem \ref{thm7}. This completes the proof.
\end{proof}

\vspace{05mm} \noindent \footnotesize
	\begin{minipage}[b]{90cm}
		\large{Department of Mathematics,\\
			School of Applied Sciences, \\ 
			Kalinga Institute of Industrial Technology University, \\ 
			Bhubaneswar, 751024, Odisha, India. \\
            Email: bptbibhu@gmail.com \\
		Email: asutoshsatapathy95@gmail.com\\
        Email: utkal.duttafma@kiit.ac.in}
	\end{minipage}	
	
	\vspace{05mm} \noindent \footnotesize
	\begin{minipage}[b]{90cm}
		\large{P.G. Department of Mathematics,\\
			Government Women's College, Sundargarh,\\ 
			Sambalpur University, Odisha, India. \\
			Email: iiit.bijan@gmail.com}
	\end{minipage}	
	
\end{document}